\documentclass[twoside,11pt]{amsart}
\usepackage{amsmath,latexsym,amssymb,times,mathptm,enumerate,verbatim,amsbsy, graphicx}
\usepackage{hyperref}

\theoremstyle{plain}
\newtheorem{thm}{Theorem}[section]
\newtheorem{lem}[thm]{Lemma}
\newtheorem{prop}[thm]{Proposition}
\newtheorem{cor}[thm]{Corollary}
\newtheorem{ques}[thm]{Question}

\numberwithin{equation}{thm}

\theoremstyle{definition}

\newtheorem{ex}[thm]{Example}
\newtheorem{ex-notn}[thm]{Example/Notation}

\newtheorem{conj}[thm]{Conjecture}

\newtheorem{rem}[thm]{Remark}

\def\gr{\operatorname{gr}}

\def\rd{\operatorname{r}}

\def\bar#1{\overline{#1}}

\def\Mtwo{{\em Macaulay} 2\expandafter}

\def\dim{{\rm dim} \,}
\def\m{{\mathfrak m}}
\def\n{{\mathfrak n}}
\def\a{{\mathfrak a}}

\def\len{{\lambda}}

\DeclareMathOperator{\eh}{e}
\DeclareMathOperator{\ord}{ord}

\begin{document}

\title{A generalization of an inequality of Lech relating multiplicity and colength}

\author[C. Huneke]{Craig Huneke}
\address{Department of Mathematics, University of Virginia, Charlottesville, VA 22903}
\email{huneke@virginia.edu}

\author[I.Smirnov]{Ilya Smirnov}
\address{Department of Mathematics, University of Michigan, Ann Arbor, MI 48109}
\email{ismirnov@umich.edu}

\author[J. Validashti]{Javid Validashti}
\address{Department of Mathematics, Cleveland State University, Cleveland, OH 44115}
\email{j.validashti@csuohio.edu}



\thanks{The first author was partially supported by NSF grant DMS-1460638, and thanks them for their support. }

\date{}

\dedicatory{Dedicated to Professor Gennady Lyubeznik on the occasion of his 60th birthday}

\begin{abstract}
We study conjectured generalizations of a formula of Lech which relates the multiplicity of a finite 
colength ideal in an equicharacteristic local ring to its colength, and prove one of these generalizations
involving the multiplicity of the maximal ideal times the finite colength ideal. We also
propose  a Lech-type formula that relates multiplicity and the number of generators.
We prove the conjecture in dimension three and establish a weaker result in full generality.
\end{abstract}
\maketitle

\section{Introduction}

A  classical inequality due to Lech in 1960 (\cite{Lech}) states that if $(R,\m)$ is a Noetherian local ring with maximal ideal $\m$ and dimension $d$, then for an $\m$-primary ideal $I$ in $R$,
\[
\eh(I) \leq d! \,  \lambda(R/I) \eh(R),
\]
where $\eh(I)$ is the Hilbert-Samuel multiplicity of $I$, the term $\lambda(R/I)$ is the length of $R/I$, and $\eh(R)$ denotes the multiplicity of the local ring $R$, i.e., the Hilbert-Samuel multiplicity of the maximal ideal $\m$. The main result of this paper  (Theorem \ref{main})  is a proof of
an strengthened inequality,
namely that if $R$ has dimension $d\geq 4$, then
\[\eh(\m I) \leq d! \,  \lambda(R/I) \eh(R). \]

Mumford used an asymptotic version of Lech's
Inequality in \cite{Mum} to give a local version of semistability/stability. The first author of this
paper suggested
a generalization of Lech's inequality which was studied by Ananthnarayan and the last author of
this paper in \cite{HV},  and proved in the same paper in several
cases.  The generalized inequality is the following. Here, and for the rest of this paper, set $P(x) = x(x+1) \cdots (x+d-1) $. 

\begin{ques}\cite{HV}\label{conj} Let $(R,\m)$ be a Noetherian local ring of dimension $d$, and let $I$ be an $\m$-primary
ideal. Is
\[
P( \eh(I)^{\frac{1}{d}} )\leq d! \,  \lambda(R/I) \eh(R)?
\]
\end{ques}

Note that $P( x^{\frac{1}{d}})-x$ grows as $x^{\frac{d-1}{d}}$ with respect to $x$, therefore this question is a strong generalization
of Lech's inequality. Moreover, the proposed inequality is sharp in a regular local ring, since  equality holds for  powers of the maximal ideal in this case, as $d! \, \lambda (R/\m^n) = d! \binom{n+d-1}{d} = P(n)= P(\eh(\m^n)^{\frac{1}{d}})$.

In \cite{HV}, several other inequalities were proposed which are successively weaker, and which involve
mixed multiplicities. We are able to prove one of these inequalities (Conjecture \ref{length conj}) in dimension at most three (Corollaries \ref{2.5dim2}~and~\ref{2.5dim3}).  The work in \cite{HV}  also led to a stronger, but natural, inequality than that of Lech in dimension at least 4 (Conjecture \ref{mI}) that is proved in Theorem \ref{main}.

The contents of the paper are organized as follows: in Section~\ref{prelim} we gather background information and some basic lemmas which is used throughout the paper.
In Section~\ref{reg-red} we reduce the proof of our main result to the regular case. In Section~\ref{Dimen 2} we prove a strong inequality for
regular local rings of dimension two, which is then used in Section~\ref{Dimen 3} to prove inequalities for three dimensional
regular local rings. Using the results for three dimensional rings, we then prove our main theorem in Section \ref{mainresult}. In a final section we treat some related conjectures on
the number of generators of integrally closed ideals and their relationship with Lech-type
inequalities. 

\bigskip

\section{Preliminaries}\label{prelim}

\subsection{Mixed multiplicities}

We begin with a summary of information concerning mixed multiplicities, which play an important
role in studying generalizations of Lech's inequality. We refer to \cite{S} for more background information. Let $(R, \m)$ be a Noetherian local ring of dimension $d$ and let $I$ and $J$ be $\m$-primary ideals. The theory of mixed multiplicities originates in \cite{Bhattacharya}, where 
Bhattacharya studied the mixed Hilbert-Samuel function $\len (R/I^sJ^t)$
and showed that it is eventually polynomial. In \cite[\S 2]{T} the {\it mixed multiplicities}  of $I$ and $J$
were defined as normalized coefficients of the highest degree terms of their mixed Hilbert-Samuel polynomial. These numbers are denoted  by $\eh_i(I\mid J)$ for $i=0, \ldots, d$, and they
 satisfy the {\it expansion formula}
\begin{equation}\label{expansion}
\eh(I^sJ^t) = \sum_{i = 0}^d \binom{d}{i} s^i t^{d - i} \eh_i (I \mid J)
\end{equation}
for all non-negative integers $s$ and $t$.
It follows that $\eh_i (I^s \mid J^t) = s^i t^{d - i}\eh_i (I \mid J)$.
The reader should be warned that our notation for mixed multiplicities is slightly different from \cite{RS}, in which the subscript $i$ is used to indicate the degree of {\it the second argument}. We establish a notation that is more in line with that of \cite[\S 2]{T}. To be precise, our $\eh_i (I \mid J)$ is $\eh_{d - i} (I \mid J)$ of \cite{RS}.  In particular, $\eh_0 (I \mid J) = \eh(J)$ and $\eh_d (I \mid J) = \eh(I)$ in our notation. 

\begin{rem}\label{mixedmultASmult}
Risler-Teissier (\cite{T}, also a generalization of Rees, \cite{Rees}) proved that if the residue field is infinite, then
\[
\eh_i (I \mid J) = \eh \left((x_1, \ldots, x_{i}, y_1, \ldots, y_{d - i}) \right),
\]
where $x_1, \ldots, x_{i}$ are general elements in $I$ and $y_1, \ldots, y_{d - i}$ are general elements in $J$. Recall that a {\it general element} in an ideal is a general linear combination of fixed generators of the ideal.
\end{rem}

\begin{rem}\label{mod form}
Let $(R,\m)$ be a Noetherian local ring with infinite residue field.
If $x_1, \ldots, x_i$ are general elements in $\m$, 
we use $R_i$ to denote $R/(x_1, \ldots, x_i)$ and $I_i$ to denote $IR_i$.  
With this notation, $\eh_i (\m \mid I) = \eh(I_i)$ for $i=0, \ldots, d-1$, since $x_1, \ldots, x_i$ are superficial (\cite[p. 306]{T}, \cite[Theorem~1.11]{S}).
\end{rem}

\subsection{Integrally closed ideals}

Recall that an element $x\in R$ is {\it integral} over an ideal $I$ if it is a root of a
polynomial $f(T)\in R[T]$ of the form $f(T) = T^n+a_1T^{n-1}+\cdots + a_n$ with $a_j\in I^j$ for
all $j=1, \ldots, n$. The set of all integral elements over $I$ forms another ideal, $\overline{I}$, the
{\it integral closure} of $I$. If $I = \overline{I}$, then $I$ is said to be integrally closed. For general information concerning integral closures we refer to \cite{SH}. 

We also refer the reader to  the theory of $\m$-full ideals which was developed by Junzo Watanabe in \cite{W}, and which shares many of the same properties and could be important for further progress. In particular, every integrally closed ideal of positive height is $\m$-full (\cite{Goto}). 
In the following, we collect some properties of integrally closed ideals or $\m$-full ideals from  \cite[14.1]{SH} or \cite[Theorem~2]{W}.  Note that $\mu(I)$ denotes the minimal number of generators of an ideal $I$. 

\begin{thm}\label{properties} 
Let $(R,\m)$ be a Noetherian local ring with infinite residue field and $I$ an $\m$-primary integrally closed ideal. Then for a general $x\in \m$,
\begin{enumerate}
\item $I : x=I : \m$,
\item $\mu(I) = \len (R_1/\m_1 I_1) = \mu(I_1) + \len (R_1/I_1)$.
\end{enumerate}
\end{thm}

\medskip

\subsection{Generalized Lech-type inequalities}

Fix positive constants $s_{d, i}$ such that
\[
P_d(n):=n(n + 1) \cdots (n+d-1) = \sum_{i = 0}^{d-1} s_{d, i} n^{d - i}.
\]
These numbers are known as unsigned Stirling numbers of the first kind. When $d$ is fixed, we shall delete the dimension subscript.
For example, if $d = 4$, then  $s_0=1$, $s_1= 6$, $s_2 = 11$, and $s_3 = 6$.
The following conjecture was made in \cite{HV},

\begin{conj}\label{length conj}
Let $(R, \m)$ be a Noetherian local ring of dimension $d$.
Then  for all $\m$-primary ideals $I$
\[\sum_{ i = 0}^{d - 1} s_i \eh_i (\m \mid I)\leq d!  \len (R/I) \eh(R). \]
\end{conj}

This conjecture is weaker than  Question~\ref{conj} proposed in the introduction. The reason it is weaker is that  we may reduce Conjecture \ref{length conj} to the regular case (Theorem \ref{reduction}), and then we compare the sum $\sum_{ i = 0}^{d - 1} s_i \eh_i (\m \mid I)$ with  $P(\eh(I)^{\frac{1}{d}}) = 
\sum_{i = 0}^{d-1} s_i  \eh(I)^{\frac{d-i}{d}}$ term by term, as  
$\eh_i (\m \mid I ) \leq   \eh(\m)^{\frac{i}{d}}   \eh(I)^{\frac{d-i}{d}}=  \eh(I)^{\frac{d-i}{d}}  $ by inequalities proved by Rees and Sharp \cite[Corollary 2.5]{RS}. Note that Conjecture \ref{length conj} is sharp, since equality holds for $I=\m$.  We  prove
Conjecture \ref{length conj} up to dimension three (Corollaries \ref{2.5dim2}~and~\ref{2.5dim3}). 

The following conjecture was also proposed in \cite{HV},

\begin{conj}\label{mI}
Let $(R, \m)$ be a Noetherian local ring of dimension $d\geq 4$.
Then  for all $\m$-primary ideals $I$
\[ \eh(\m I) \leq d! \len (R/I)\eh(R). \]
\end{conj}

\smallskip

This conjecture too is weaker than  Question~\ref{conj}, since by the expansion formula (\ref{expansion}) we may write
$\eh(\m I) = \sum_{ i = 0}^{d} \binom{d}{i} \eh_i (\m \mid I).$
This sum can then be compared to $P(\eh(I)^{\frac{1}{d}}) = \sum_{ i = 0}^{d - 1} s_i \eh(I)^{\frac{d-i}{d}}$, as was done
in \cite{HV}. 
The main result of this paper is a proof  of Conjecture \ref{mI} (Theorem \ref{main}). 
 \smallskip
\subsection{A Basic Lemma}

In this section we prove a lemma which we use throughout this paper.

\begin{lem}\label{colonmult}
Let $(R, \m)$ be a Noetherian local ring of dimension $d$.
Assume that $I\subseteq R$ is $\m$-primary and $x\in R$ is a non-zero divisor. Let $\m_1$ and $I_1$ denote the images of $\m$ and $I$ in the  ring $R_1=R/(x)$ and set $J=I:x$.  Then

\begin{equation}\label{length}
\lambda (R/I) = \lambda (R/J) +\lambda( R_1/I_1),
\end{equation}
\begin{equation}\label{mon}
\eh(I) \leq \eh(J) +d \eh(I_1).
\end{equation}
\end{lem}
\begin{proof}
Consider the short exact sequence
$$
0\longrightarrow \frac{R}{I^n:x^n} \longrightarrow \frac{R}{I^n} \longrightarrow \frac{R}{I^n+(x^n)} \longrightarrow 0
$$
where $n$ is any positive integer. Taking length along the above sequence we obtain
$$
\lambda (R/I^n) = \lambda (R/(I^n:x^n)) +\lambda( R/(I^n + (x^n)))
$$
Letting $n=1$ we obtain  the Equation \ref{length}. Note that $J^n=(I:x)^n\subseteq I^n:x^n$, thus 
$$\lambda (R/(I^n:x^n)) \leq   \lambda (R/J^n).$$ 
We also have
$$\lambda( R/(I^n + (x^n)) \leq n  \lambda( R/(I^n + (x)) = n \lambda( R_1/I_1^n).$$ 
To see this, observe first that for an element $x$ in an Artinian local ring $A$ 
\[
\lambda( A/(x^n)) = \lambda(A/( x^{n-1})) + \lambda( (x^{n-1})/(x^n))
\leq \lambda(A/( x^{n-1})) + \lambda (A/(x)). 
\]
Then by induction we may derive that $\lambda (A/(x^n)) \leq n \lambda(A/(x))$, and apply this in $A = R/I^n$.

Therefore
$$
\lambda (R/I^n) = \lambda (R/(I^n:x^n)) +\lambda( R/(I^n + (x^n)) \leq  \lambda (R/J^n)+ n  \lambda( R_1/I_1^n).
$$
Now comparing the leading coefficients of the polynomials arising on both sides for $n$ large, Inequality \ref{mon} follows. 
\end{proof}

\medskip

\section{Reduction to regular local rings}\label{reg-red}

In this section, we show that generalized   Lech-type inequalities as in Conjectures~\ref{length conj}~and~\ref{mI} 
may be reduced to the case of regular local rings.
\begin{thm}\label{reduction}
Let $d$ be a positive integer and $n_0, \ldots, n_d$ non-negative real numbers. Assume that
for every complete regular local ring $(S, \n)$ of dimension $d$ with infinite residue field and for all integrally closed $\m$-primary ideals $\a$ in $S$,
\begin{equation}\label{formula-regular}
\sum_{i = 0}^d n_i \eh_i (\n \mid \a) \leq d!\len(S/\a).
\end{equation}
Then for every Noetherian local ring $(R, \m)$ of dimension $d$ and for all $\m$-primary ideals $I$ in $R$, 
\begin{equation}\label{formula-general}
\sum_{i = 0}^d n_i \eh_i (\m \mid I) \leq d!\len(R/I) \eh(R).
\end{equation}
\end{thm}
\begin{proof}
We can extend the ground field $R/\m$ to be infinite by a faithfully flat extension
which does not change length, multiplicity, or dimension. 
Let $G: = \gr_{\m}(R)$ be the associated graded ring of $R$ whose $n$th graded piece is $\m^n/\m^{n+1}$.
We denote the unique maximal homogeneous ideal of $G$ by $\mathfrak{M}$.
By definition, the multiplicities  of $G$ and $R$ are equal. 
For an arbitrary ideal $I$ of $R$ let $I^*$ be the {\it form ideal} of $I$, 
\[
I^* := \bigoplus_{ n \geq 1} \frac{I \cap \m^n + \m^{n + 1}}{\m^{n + 1}},
\]
the ideal generated by the leading forms in $G$ of all elements of $I$. 
Observe that $\gr_{\m/I}(R/I)\cong G/I^*$. Thus, if $I$ is $\m$-primary, then $\lambda (R/I)= \lambda (\gr_{\m/I}(R/I)) = \lambda (G/I^*) $. In addition, $\eh(I) \leq \eh(I^*)$, since
$(I^*)^n \subseteq (I^n)^*$ for all $n \geq 1$.

Let $x_1,\ldots, x_i \in \m$ be general linear froms. Then $x_1^*, \ldots, x_i^* \in \mathfrak{M}$ are general as well. For an $\m$-primary ideal  $I$ and general elements  $\alpha_1, \ldots, \alpha_{d-i} \in I^*$, choose $l_1, \ldots, l_{d-i} \in I$ such that $l_j^*=\alpha_j$ for $j=1, \ldots, d-i$. Let $L=(x_1,\ldots, x_i, l_1, \ldots, l_{d-i})R$. Then  by \cite[Lemma~2.8]{S} and Remark \ref{mixedmultASmult}, we obtain
$$
\eh_i(\m \mid I) \leq \eh(L) \leq \eh( L^* )  \leq \eh((x_1^*, \ldots, x_i^*, \alpha_1,\ldots, \alpha_{d-i} )G) = \eh_i(\mathfrak{M} \mid I^*). 
$$

Therefore, (\ref{formula-general}) descends from $G$ to $R$, 
\[
\sum_{i = 0}^d n_i \eh_i (\m \mid I) \leq \sum_{i = 0}^d n_i \eh_i (\mathfrak{M} \mid I^*) \leq d!\len(G/I^*) \eh(G) = d!\len(R/I) \eh(R).
\]

We change notation and assume that $R$ is a standard graded ring over
an infinite field \textsf{k}, with maximal homogeneous ideal $\m$, and let $I$ be an $\m$-primary homogeneous ideal. 

We now proceed as in \cite[Corollary~3.9]{Mum}.  
For general linear forms $t_1, \ldots, t_d \in \m$, the ideal $(t_1, \ldots, t_d)$ is a minimal reduction of $\m$
and,  by the Noether normalization, $R$ is a finitely generated graded module over a polynomial ring $S:=\textsf{k}[t_1, \ldots, t_d]$. 
Set $\n =(t_1, \ldots, t_d)S$ and $\a= I \cap S$. Let $x_1,\ldots, x_i \in \n$ and $y_1,\ldots, y_{d-i} \in \a$ be general linear froms. Let $\mathfrak{b}=(x_1,\ldots, x_i, y_1,\ldots, y_{d-i})S$.
Since the rank of $R$ as an $S$-module is $\eh(R)$, we have  
$\eh (\mathfrak{b} R) = \eh(\mathfrak{b}) \eh(R)$.
Then by 
\cite[Lemma~2.8]{S} and Remark \ref{mixedmultASmult} we obtain
$$ 
\eh_i(\m \mid I) \leq
\eh(\mathfrak{b}R)=
\eh(\mathfrak{b})\eh(R)
= \eh_i(\n \mid \a) \eh(R).
$$

Since $S/\a \subseteq R/I$ we have 
\[ \len_S (S/\a) = \dim_k S/\a \leq \dim_k R/I = \len_R(R/I). \]
Therefore, (\ref{formula-general}) ascends from $S$ to $R$, 
\[
\sum_{i = 0}^d n_i \eh_i (\m \mid I) \leq \sum_{i = 0}^d n_i \eh_i (\n \mid \a) \eh(R) \leq
d! \len (S/\a)\eh(R) \leq d!\len(R/I) \eh(R).
\]

We have reduced (\ref{formula-general})
to  the case in which $R$ is a polynomial ring. We may also complete at the unique homogeneous
maximal ideal and reduce to the case that $S$ is a complete regular local ring of dimension $d$ with infinite residue field.

Finally, it is enough  to consider integrally closed ideals  in $S$, since the left-hand side of (\ref{formula-regular}) remains the same if we replace $\a$ with its integral closure, while the length in the right-hand side can only decrease.

\end{proof}

\medskip
\section{Regular Local Rings of Dimension Two}\label{Dimen 2}

In order to prove our generalized Lech inequalities in Conjectures \ref{length conj}~ and~\ref{mI}, we first need to find inequalities in dimensions two and three which are
finer than Lech's original inequality. In this section, when $R$ is a $2$-dimensional regular local ring  we  use classical results on the properties of integrally
closed ideals in such rings. We need the following notation. 
Let $(R,\m)$ and $(R',\m_{R'})$ be two-dimensional regular local rings. We say that $R'$ birationally dominates $R$ if $R \subseteq R'$, $\m_{R'} \cap R = \m$ and $R$ and $R'$ have the same quotient field. We denote this by $R \leq R'$. Let $[R':R]$ denote the degree of the field extension $R/\m \subseteq R'/\m_{R'}$.
Further if $I$ is an $\m$-primary ideal in $R$, let $I^{R'}$ be the ideal in $R'$ obtained from $I$ by factoring $I R' = x I^{R'}$, where $x$ is the greatest common divisor of the generators of $I R'$. 
The following theorem (\cite[Theorem 3.7]{JV}) gives a formula for $\eh(I)$. Recall that $\ord(I)$ denotes the largest integer $r$ such that $I \subset \m^r$.

\begin{thm}[Multiplicity Formula]\label{multiplicityFormula}
Let $(R,\m)$ be a two-dimensional regular local ring and $I$ be an $\m$-primary ideal. Then $$\eh(I) = \sum_{R \leq R'} [R':R]\ord(I^{R'})^2.$$ 
\end{thm}

The following formula (\cite{JV}, Theorem 3.10) is attributed to Hoskin and Deligne.

\begin{thm}[Hoskin-Deligne Formula]\label{HD}
Let $R$, $I$ be as in Theorem \ref{multiplicityFormula}. Further assume that $I$ is an integrally closed ideal. Then, $$\lambda(R/I) = \sum_{R \leq R'}{\ord(I^{R'})+1\choose 2}[R':R].$$
\end{thm}

We also need a formula of Lipman (\cite[Lemma 2.2]{Lipman}). This first requires a definition.
If $R\leq R'$ and $R\ne R'$, then $R'$ is said to be \it proximate \rm to $R$ if the valuation ring $V$ of the
order valuation of $R$ contains $R'$. We write $R'\succ R$ in this case. 

\begin{thm}\label{Proximity} Let $(R,\m)$ be a two-dimensional regular local ring, and let $I$
be an $\m$-primary integrally closed ideal. Set $r(I)$ equal to the largest integer $r$ such that
$I = \m^rJ$ for some ideal $J$ (allowing $J = R$). Then
$$r(I) = \ord_R(I) - \sum_{R' \succ R}\ord(I^{R'})[R':R].$$
\end{thm}

Putting these results together gives us a fairly sharp upper bound for $\eh(I)$ in the case that
$I$ is an 
$\m$-primary ideal in a two-dimensional regular local ring.  Recall that if $I$ is an ideal, then $\overline{I}$  denotes the integral closure
of $I$.

\begin{thm}\label{dim2}
Let $(R,\m)$ be a two-dimensional regular local ring and $I$ be an $\m$-primary ideal. Then 
$$\eh(I) \leq 2\cdot \len(R/I) - 2\cdot\ord_R(I) + r(\overline{I}).$$ 
\end{thm}

\begin{proof}
We first observe that we may assume $I$ is integrally closed. This is due to the fact that the left-hand side of the inequality does not change, while
in the right-hand side the length can only decrease, and the order stays the same. Using the Hoskin-Deligne formula and the multiplicity formula, we see that
$$\len(R/I) = \sum_{R \leq R'}\frac{(\ord(I^{R'})+1)\ord(I^{R'})}{2}[R':R] = \frac{\eh(I)}{2} + \frac{1}{2}\sum_{R \leq R'}\ord(I^{R'})[R':R].$$
Hence 
$$
2\cdot\len(R/I) = \eh(I)+\sum_{R \leq R'}\ord(I^{R'})[R':R]\geq \eh(I)+\ord_R(I) +\sum_{R' \succ R}\ord(I^{R'})[R':R],
$$
and therefore by Theorem \ref{Proximity},
$$2\cdot\len(R/I) \geq  \eh(I) + 2\cdot\ord_R(I) - r(I),$$
as claimed.  \end{proof}

\begin{cor}\label{2.5dim2}
Conjecture \ref{length conj} holds in dimension two.
Namely, if $I$ is an $\m$-primary ideal in a two-dimensional Noetherian  local ring $(R,\m)$, 
then 
\begin{equation}\label{2.5dim2ineq}
\eh(I) + \eh_1 (\m \mid I)\leq 2\len(R/I)\eh(R).
\end{equation}
Moreover, if $R$ is regular, then equality holds if and only if  $I$ is a power of the maximal ideal.  
\end{cor}

\begin{proof} By Theorem \ref{reduction} we may assume that $R$ is a two-dimensional  regular local ring with infinite residue field and $I$ is integrally closed.  Using  Remark \ref{mod form} we observe that  $\eh_1 (\m \mid I)=\eh (I_1)=\ord_R(I) \geq r(I)$ since $R_1$ is a DVR.
Therefore,  (\ref{2.5dim2ineq}) follows immediately from Theorem \ref{dim2}. Furthermore,  since 
$$
\eh(I) + \eh_1 (\m \mid I) =\eh(\bar{I}) + \eh_1 (\m \mid \bar{I}) \leq 2\len(R/\bar{I})\eh(R) \leq 2\len(R/I)\eh(R),
$$
 equality in (\ref{2.5dim2ineq}) implies $\bar{I}=I$. Therefore, if $R$ is regular and  equality holds in (\ref{2.5dim2ineq}), then  $r(I)= \ord(I)$ by Theorem \ref{dim2}, hence $I$ is a power of the maximal ideal.
\end{proof}

\medskip
\section{Regular Local Rings of Dimension Three}\label{Dimen 3}

\medskip


In order to prove our main theorem improving Lech's inequality in dimension four or higher, we need 
an improvement in dimension three which is as precise as possible. A fairly tight formula as was given
in the last section for two-dimensional regular local rings is probably impossible to attain in dimensions
at least three. We need an improvement which is not so precise that we cannot prove it, but is robust
enough to allow reductive steps in dimension four to go through. This means the inequality is
a little delicate.

Let $(R,\m)$ be a regular local ring with infinite residue field. 
Let $x$ denote a general element of $\m$, i.e., a general linear combination of fixed generators of $\m$. 
As in the preliminaries, by $R_1$ we denote $R/(x)$, and by $I_1$ we denote the image of an ideal $I$ in $R_1$. Similarly,
by $R_2$ we denote the ring obtained from $R$ by moding out the ideal generated by two general linear elements, and let $I_2$ be the image of $I$ in $R_2$. 


The main result in this section is the following more precise estimate giving an inequality
between various multiplicities and colength.

\begin{thm}\label{dim3thm} 
Let $(R, \m)$ be a three-dimensional regular local ring with infinite residue field and 
$I$ be an  $\m$-primary ideal. Then for all $c \in [0, 2]$ we have 
\[
\eh(I) + 3\eh (I_1) + (2 + c) \eh(I_2) - c\rd(\overline{I_1})\leq 6 \len (R/I).
\]
\end{thm}
\begin{proof} 
We can rewrite the assertion as 
\[
\eh(I) + 3\eh (I_1) + 2 \eh(I_2) + c(\eh(I_2) - \rd(\overline{I_1}) )\leq 6 \len (R/I).
\]
Since $R_2$ is a DVR, we have $\eh(I_2) =  \ord (I_2)= \ord (I_1)$. Also note that  $\ord (I_1) = \ord (\overline{I_1})$, since  $R_1$ is a regular local ring and powers of $\m_1$ are integrally closed. Thus  $\eh(I_2) =  \ord (\overline{I_1}) \geq \rd(\overline{I_1})$. Therefore, it is enough to prove the inequality for $c = 2$.

We use induction on $\len(R/I)$, the base case of $I = \m$ is clear as both sides equal $6$.
We may assume that $I$ is integrally closed. Choose a general linear form $x \in \m$ and let $J = I :x$.
By Theorem~\ref{properties} $I : x = I :\m$, so $\m J \subseteq I$.
Therefore, by the expansion formula (\ref{expansion}) for  $\eh(\m J)$, 
\[
\eh(I) \leq \eh(\m J) =\eh(J) + 3 \eh(J_1) + 3 \eh(J_2) + 1,
\] 
and after applying the induction hypothesis to $J$ we refine this to
\begin{equation}\label{EQ1}
\eh(I) \leq 6 \len (R/J) - \eh(J_2) + 2\rd(\overline{J_1})+ 1.
\end{equation}

By Theorem~\ref{dim2} $\eh(I_1) \leq 2\len (R_1/I_1) - 2\eh(I_2) + \rd(\overline{I_1})$.
Thus we derive that
$$
\eh(I) + 3\eh(I_1) + 4\eh(I_2) - 2\rd(\overline{I_1}) 
\leq \eh(I) + 6\len (R_1/I_1) -2 \eh(I_2) + \rd(\overline{I_1}).
$$
After adding the estimate of $\eh(I)$ from (\ref{EQ1}) and using Lemma~\ref{colonmult}, we obtain that
$$
\eh(I) + 3\eh(I_1) + 4\eh(I_2) - 2\rd(\overline{I_1}) 
\leq 
6\len(R/I)  - \eh(J_2) + 2\rd(\overline{J_1})+ 1 -2 \eh(I_2) + \rd(\overline{I_1}).
$$
Thus the assertion of the theorem follows when we can show that
\[
 - \eh(J_2) + 2\rd(\overline{J_1})+ 1 - 2 \eh(I_2) + \rd(\overline{I_1}) \leq 0,
\]
or since $R_2$ is a DVR, 
\[
- \ord(J_1) + 2\rd(\overline{J_1})+ 1 - 2 \ord(I_1) + \rd(\overline{I_1}) \leq 0.
\]

We regroup the terms in this inequality and we write it as 
\begin{equation}\label{EQ2}
 \left(\ord(I_1)- \ord(J_1) \right) + \left( \ord (I_1)  - \rd(\overline{I_1})\right) + 2\left(\ord(J_1)  - \rd(\overline{J_1})\right) \geq 1.
\end{equation}
Since $ \ord(I_1) \geq \rd(\overline{I_1})$  and $\ord(I_1) \geq \ord(J_1) \geq \rd(\overline{J_1})$,
Inequality \ref{EQ2} certainly holds unless 
\[ \ord(I_1) = \rd(\overline{I_1}) = \ord(J_1) = \rd(\overline{J_1}).\]
Then  $\overline{I_1} = \overline{J_1}= \m_1^r$ for some $r$.
Thus $\eh(I_1) = \eh(J_1)$ and $\eh(I_2) = \eh(J_2)$. Therefore, 
\[
\eh(I) + 3\eh(I_1) + 4 \eh  (I_2) - 2\rd(\overline{I_1}) 
= \eh(I) + 3\eh(I_1) + 2\eh (I_2).
\]
By Lemma~\ref{colonmult} and the induction hypothesis,
\[
 \eh(I) + 3\eh(I_1) + 2\eh (I_2)
\leq \eh(J) + 6\eh(J_1) + 2\eh(J_2) \leq  6\len (R/J)+ 3\eh(J_1).
\]
However, $3\eh(J_1)= 3\eh(I_1) < 6 \len (R_1/I_1)$ by Lech's Inequality in dimension 2. Therefore,  
$$ 6\len (R/J) + 3\eh(J_1) <  6\len (R/J) + 6 \len (R_1/I_1)= 6 \len(R/I)$$
by Lemma~\ref{colonmult} and the theorem follows. 
\end{proof}

\begin{cor}\label{2.5dim3}
 Conjecture~\ref{length conj} holds in dimension three. That is, if $I$ is an $\m$-primary ideal in a Noetherian  local ring $(R,\m)$ of dimension three, 
then 
\begin{equation}\label{2.5dim3ineq}
\eh(I) + 3 \eh_1 (\m \mid I) +  2\eh_2 (\m \mid I) \leq 6\len(R/I)\eh(R).
\end{equation}
Moreover, if $R$ is regular, then equality holds if and only if  $I$ is a power of the maximal ideal.  
\end{cor}
\begin{proof}
By Theorem \ref{reduction} we may assume that $R$ is a three-dimensional regular local ring with infinite residue field. Then $\eh_{i}(\m \mid I) = \eh(I_i)$ for $i=0,1,2$  by Remark \ref{mod form},  and  
Inequality \ref{2.5dim3ineq} corresponds to the case $c=0$ in Theorem \ref{dim3thm}.

We use induction on $\lambda(R/I)$ to prove that equality in (\ref{2.5dim3ineq}) implies $I$ is a power of $\m$ in the regular case. The assertion is trivially true for the base case $\lambda(R/I)=1$. Suppose 
\begin{equation}\label{E1}
\eh(I) + 3 \eh_1 (\m \mid I) +  2\eh_2 (\m \mid I) = 6\len(R/I).
\end{equation}
Then $I$ is integrally closed, since
$$
\eh(I) + 3 \eh_1 (\m \mid I) +  2\eh_2 (\m \mid I) = \eh(\bar{I}) + 3 \eh_1 (\m \mid \bar{I}) + 2\eh_2 (\m \mid \bar{I}) \leq 6\len(R/\bar{I}) \leq 6\len(R/I).
$$
We may further  assume that the residue field is infinite. Let $J = I :x$ for a general linear form $x \in \m$. Therefore $\m J \subseteq I \subseteq J$ by Theorem~\ref{properties}. Using the expansion formula (\ref{expansion}) for $\eh(\m J)$,
\begin{equation}\label{E2}
\eh(I) \leq \eh(\m J)=\eh(J) + 3 \eh(J_1) + 3 \eh(J_2) + 1.
\end{equation}
In addition,  by Theorem \ref{dim3thm},
\begin{equation}\label{E3}
\eh(J) + 3\eh (J_1) + 2 \eh(J_2) \leq 6 \len (R/J).
\end{equation}
If one of the  Inequalities \ref{E2} and \ref{E3} is strict, then  adding them and subtracting 1 from the right-hand side, we obtain
\begin{equation}\label{E4}
\eh(I) \leq 6 \len (R/J)  + \eh(J_2).
\end{equation}
By Corollary \ref{2.5dim2} we have $3\eh(I_1) + 3\eh(I_2) \leq  6\len (R_1/I_1)$. Therefore, adding this  inequality to (\ref{E4}) and using Lemma~\ref{colonmult}, we derive
\begin{equation}\label{E5}
\eh(I) + 3\eh(I_1) + 3\eh(I_2) \leq 6 \len (R/J)  + \eh(J_2)  +6\len (R_1/I_1)
\leq  6\len(R/I)  + \eh(J_2).
\end{equation}
Subtracting (\ref{E1}) from Inequality \ref{E5} we obtain  $\eh(I_2) \leq \eh(J_2)$. This implies $\eh(I_2)$ and $\eh(J_2)$ are equal, since $\eh(I_2) \geq \eh(J_2)$ as $I_2\subseteq J_2$.
Therefore,  $\ord(I_1) = \eh(I_2)=\eh(J_2) = \ord(J_1)$, since $R_2$ is a DVR. On the other hand, (\ref{E1}) implies $\ord(I_1)= \rd(\overline{I_1})$ by Theorem \ref{dim3thm}. Thus $\bar{I_1}=\m_1^r$ for some $r$. Hence, $\ord(J_1)=\ord(I_1)=r$, so $I_1\subseteq J_1 \subseteq \m^r$, and we conclude $\bar{I_1}=\bar{J_1} = \m_1^r$. However, in this case  $\eh(I) + 3\eh (I_1) + 2 \eh(I_2)$ is strictly less than $6 \len (R/I)$, as it is shown in the proof of Theorem  \ref{dim3thm}, which contradicts   (\ref{E1}).
Therefore, equality must hold in both Inequalities \ref{E2} and \ref{E3}. The equality $\eh(I) = \eh(\m J)$ implies  $I=\bar{\m J}$ due to a well-known result of Rees (\cite{Rees1}, \cite[Theorem 11.3.1]{SH}).  On the other hand, equality in  (\ref{E3}) implies $J$ is a power of $\m$ by induction.  Therefore, $I$ is a power of $\m$.
\end{proof}

\medskip

\section{The main result}\label{mainresult}

In this section we are able to use the previous results to prove our
main result in Theorem \ref{main}.
An induction reduces this theorem to the case of regular local rings
of dimension four. Here the
result is subtle, and we use the somewhat strange estimates
for dimension three
regular local rings developed in Theorem
\ref{dim3thm} to finish the proof in dimension four.

\begin{thm} \label{main}
Let $(R,\m)$ be a Noetherian local ring of dimension  $d \geq 4$. Let $I$ be an
$\m$-primary ideal. Then 
\begin{equation}\label{maininequality}
\eh(\m I)\leq d!\len(R/I)\eh(R).
\end{equation}
\end{thm}

\begin{proof} 
By the expansion formula (\ref{expansion}) we may express $\eh(\m I)$ as a linear combination of the mixed multiplicities of $I$ and $\m$ with binomial coefficients. Therefore, to prove the assertion we may use Theorem \ref{reduction} to assume that $R$ is a regular local ring of dimension $d$ with infinite residue field and $I$ is integrally closed.

We use induction on $d$. We do the base case of $d = 4$ after the inductive step.
Assume that $d>4$. We also use induction on the
length of $R/I$. The base case of this second induction is when $I = \m$. In this case
$\eh(\m^2) = 2^d< d!=d!\len(R/\m)$ whenever $d\geq 4$. Let $x \in \m$ be a general element and set $J=I:x$. Since $\m J \subseteq \m I :x$ we have
$$
\eh(\m I: x) \leq \eh(\m J).
$$
Let $\m_1$ and $I_1$ denote the images of $\m$ and $I$ in  $R_1=R/(x)$.  Then, by Lemma \ref{colonmult} we have
$$
\eh(\m I)\leq  \eh(\m I: x) + d \eh(\m_1 I_1) \leq \eh(\m J) + d \eh(\m_1 I_1)
$$
By induction on dimension we have 
$$\eh(\m_1 I_1) \leq (d-1)! \len( R_1/I_1),$$
so, using Lemma~\ref{colonmult} twice and induction on colength of $I$, we obtain
$$\eh(\m I) \leq \eh(\m J) + d \eh(\m_1 I_1) \leq d! \len( R/J) + d (d-1)! \len( R_1/I_1)=d! \len (R/I).
$$

It remains to handle the case in which $d = 4$. We again use induction on the colength
of $I$, the case in which $I = \m$ is already done. 
Using the expansion formula (\ref{expansion}) for $\eh(\m I)$ in terms of mixed multiplicities and Remark \ref{mod form}, we obtain that
$$\eh(\m I)= \eh(I) + 4\eh(I_1)+6\eh(I_2)+4\eh(I_3) +1.$$
We choose a general $x\in \m$ and set $J = I:x\subseteq \m$. 
Then  $\m J\subseteq I$ by Theorem~\ref{properties}, and
so $\eh(I)\leq \eh(\m J)\leq 24\len(R/J)$, where the last inequality is from our induction
on the colength of $I$. We now apply Theorem \ref{dim3thm}  to the term $\eh(I_1)$. Combining
these, we obtain that
$$\eh(\m I)\leq 24\len(R/J) + 4[6\len(R_1/I_1)-3\eh(I_2)-2\eh(I_3)]+6\eh(I_2)+4\eh(I_3) +1=
24\len(R/I) - 6\eh(I_2)-4\eh(I_3)+1.$$
It follows that $\eh(\m I)\leq 24\len(R/I).$ \end{proof}

\begin{rem}
The proof of Theorem \ref{main} shows that Inequality \ref{maininequality} is indeed strict. Also note that  Inequality \ref{maininequality}  is not true for $d< 4$, for instance it fails for $I=\m$.
\end{rem}

\medskip

\section{Number of generators}

In this brief section we explore
another set of related conjectures originating in a paper of Dao and Smirnov (\cite{Dao-Smirnov}).
In \cite[Theorem~3.1]{Dao-Smirnov} they proved that  for an integrally closed $\m$-primary ideal $I$ in a regular local ring  $(R, \m)$ of dimension $d$ one has a Lech-like bound
\[\eh_1 (\m \mid I) \leq (d-1)! (\mu(I) - d + 1).\] 
In dimension two, this happens to be, in fact, an equality $\mu(I) - 1 = \eh_1 (\m \mid I)$. 
However, the proof of  \cite[Theorem~3.1]{Dao-Smirnov} shows that in dimension at least three the displayed inequality is never an equality. 
The following conjecture proposes a way to strengthen this inequality. Fix positive constants $t_{i, d}$ such that
\[
Q_d(n):=(n + 1) \cdots (n+d-1) = \sum_{i = 1}^{d} t_{d, i} n^{d - i}.
\]
We shall delete the dimension subscript when the dimension $d$ is fixed. Note that $P(n)=nQ(n)$, thus $t_i=s_{i-1}$ for $i=1, \ldots, d$.

\begin{conj}\label{mu conj}
Let $(R, \m)$ be a regular local ring of dimension $d$.
Then  for all $\m$-full (e.g., integrally closed) $\m$-primary ideals $I$,
\[
\sum_{ i = 1}^{d} t_i \eh_i (\m \mid I) \leq (d-1)! \mu(I).
\]
\end{conj}

Note that when $d$ is at least three, $(d-1)!(d-1)< d! -1 = \sum_{i = 2}^{d} t_i$. Therefore, using Conjecture~\ref{mu conj},
$$
 \eh_1 (\m \mid I) + (d-1)!(d-1) \leq \eh_1 (\m \mid I) + \sum_{i = 2}^{d} t_i \eh_i (\m \mid I ) \leq (d-1)! \mu(I).
$$
Hence, Conjecture~\ref{mu conj} is stronger than \cite[Theorem~3.1]{Dao-Smirnov} when the dimension is at least three.

\begin{rem}
In Conjecture~\ref{mu conj} equality holds for powers of the maximal ideal, since by
the linearity of the mixed multiplicities,
$$
\sum_{ i = 1}^{d} t_i \eh_i (\m \mid \m^n) =  \sum_{ i = 1}^{d} t_{i} n^{d-i} = Q(n) =  (d-1)! \binom{n + d - 1}{n} =  (d-1)!\mu (m^n).
$$
\end{rem}

\begin{thm}
Conjecture~\ref{length conj} in dimension $d$ implies Conjecture~\ref{mu conj} in dimension $d + 1$.
In particular, Conjecture~\ref{mu conj} holds in dimension at most $4$.
\end{thm}
\begin{proof}
Let $(R, \m)$ be a regular local ring of dimension $d + 1$ and $I$  an integrally closed $\m$-primary ideal. Let $R_1=R/(x)$ for a general element $x \in \m$. Then $R_1$ is a regular ring of dimension $d$,  and by Conjecture~\ref{length conj} in dimension $d$ and Theorem~\ref{properties}  we have
\[
\sum_{i = 0}^{d-1} s_{i} \eh_i (\m_1 \mid \m_1I_1 ) \leq d! \len (R_1/\m_1 I_1) = d! \mu(I),
\]
where $s_i$'s are the coefficients of $P_d(n)$. Note that  by  Remark~\ref{mod form} and the expansion formula (\ref{expansion}),
\[
\eh_i (\m_1 \mid \m_1I_1) = \eh(\m_{i + 1} I_{i + 1})= 
\sum_{k = 0}^{d -i} \binom {d-i}{k} \eh_{k} (\m_{i+1} \mid I_{i+1})=
\sum_{k = 0}^{d - i } \binom {d - i }{k} \eh_{k+i +1} (\m \mid I).
\]
Thus, there are coefficients $c_{i}$ such that $\sum_{i = 0}^{d-1} s_{i} \eh_i (\m_1 \mid \m_1I_1 ) = \sum_{i = 1}^{d+1} c_{i} \eh_i (\m \mid I)$
for all $I$. It remains to see that $c_{i}$'s are indeed the coefficients of $Q_{d+1}(n)$. To this end, one may compare  the polynomials arising on both sides when setting  $I= \m^n$. In this case, 
$$
\sum_{i = 0}^{d-1} s_{i} \eh_i (\m_1 \mid \m_1I_1 ) = \sum_{i = 0}^{d-1} s_{i} \eh_i (\m_1 \mid \m_1^{n+1} ) =
 \sum_{i = 0}^{d-1} s_{i} (n+1)^{d-i} = P_d(n+1)=Q_{d+1}(n),
$$
and
$$
\sum_{i = 1}^{d+1} c_{i} \eh_i (\m \mid I) = \sum_{i = 1}^{d+1} c_{i} \eh_i (\m \mid \m^n) =  \sum_{i = 1}^{d+1} c_{i} n^{d-i+1}.
$$
Therefore,  
$Q_{d+1}(n) = \sum_{i = 1}^{d+1} c_{i} n^{d-i+1}$
and the claim follows.
\end{proof}

A similar argument allows us to derive corollaries of Theorem~\ref{main}.

\begin{prop}\label{7.4}
Let $(R, \m)$ be a regular local ring of dimension $d\geq 5$.
Then for every $\m$-full (e.g., integrally closed) ideal $I$,
\[
 \sum_{i = 1}^d 2^{i-1} \binom{d - 1}{i - 1} \eh_i (\m \mid I) \leq (d-1)! \mu(I).
\]

\end{prop}
\begin{proof}
By Theorem~\ref{properties} and Theorem~\ref{main} we have
$$
\sum_{i = 0}^{d - 1} \binom{d - 1}{i} \eh_{i} (\m_1^2 \mid I_1) =  \eh(\m_1^2I_1) 
\leq (d - 1)!\len (R_1/\m_1 I_1) = (d - 1)!\mu (I).
$$
However, $\eh_{i} (\m_1^2 \mid I_1) = 2^i \eh_i(\m_1 \mid I_1) = 2^i \eh_{i + 1} (\m \mid I)$
and the claim follows.
\end{proof}

\begin{rem}
Under the assumptions of Theorem \ref{7.4}, following the proof of \cite[Theorem~3.1]{Dao-Smirnov}, we may similarly derive that
\[
\sum_{i = 0}^{d - 1} \binom{d-1}{i} \eh_{i} (\m_1 \mid I_1)
=  \eh (\m_1 I_1) \leq (d-1)!\len (R_1/I_1) \leq 
(d-1)! (\mu (I) - d + 1).
\]
Thus
\[
\sum_{i = 1}^{d} \binom{d-1}{i - 1} \eh_{i} (\m \mid I) \leq (d-1)! (\mu (I) - d + 1).
\]
However, this inequality is often weaker than Proposition~\ref{7.4}, as the coefficients at $\eh_i (\m \mid I)$ are smaller.
\end{rem}

\subsection{An example}
We end this section with an example that illustrates our results.

\begin{prop}\label{EX}
Let $R= \textsf{k}[[x,y,z]]$ and 
consider a family of monomial ideals $I = \overline{(x^a, y^b, z^c, xyz)}$ such that $3 \leq a \leq b \leq c$ and  $1/a + 1/b+1/c \leq 1$.
Then
\begin{enumerate}
\item $\mu (I) = 2a + b + 1$,

\item $\eh(I) = ab + bc + ac$,
\item $\eh_1 (\m \mid I) = 2a + b$,
\item $\eh_2(\m \mid I) = 3$,
\item $\len (R/I) = f(a, b) + f(b,c) + f(a, c) - a - b - c + 1$, where 
\[
f(a, b) = 
\begin{cases}
\lceil \frac{ab + b + a}{2}\rceil - 1, & \text{if $a$ does not divide $b$}\\
\lceil \frac{ab + b}{2}\rceil, & \text {if $a$ divides $b$}
\end{cases}.
\]
\end{enumerate}
\end{prop}
\begin{proof}
As a first step, we observe that
\[
I = \overline {(x^a, y^b)} + \overline {(y^b, z^c)} + \overline {(x^a, z^c)} + (xyz).
\]
We can easily compute the number of generators using this observation.
Since $\mu(\overline {(x^a, y^b)})$ can be computed in $\textsf{k}[x,y]$ using the order, 
it has $a + 1$ generators. 
Thus $\mu(I)$ is obtained by adding these generators up and noting that $x^a, y^b, z^c$ were counted twice. 

In order to compute the colength, we first note that if $a \leq b$, then  
\[
\len (\textsf{k}[x,y]/\overline {(x^a, y^b)}) = f(a,b).
\]
This can be seen by induction, using that a basis of  $\textsf{k}[x,y]/\overline{(x^a, y^b)}$ is given by monomials $x^iy^j$ such that
$0 \leq i < a$ and $0 \leq j < b - \lfloor i \frac ba \rfloor$.
Then we can compute $\len (R/I)$ by the inclusion-exclusion formula:
\[
\len (R/I) = \len (\textsf{k}[x,y]/\overline {(x^a, y^b)}) + 
\len (\textsf{k}[x,z]/\overline {(x^a, z^c)}) + \len (\textsf{k}[y,z]/\overline {(y^b, z^c)})
- a - b - c +1,
\]
where $a,b,c$ represent the number of points on each of the axes respectively. 

To compute $\eh(I)$,  $\eh_1(\m \mid I)$ and $\eh_2(\m \mid I)$, we may assume $I=(x^a, y^b, z^c, xyz)$, since 
these multiplicities are invariant up to integral closure.

We compute the multiplicity $\eh(I)$ by using that $xyz \in I$ is a superficial element. Thus by the additivity property
\[
\eh(I) = \eh(I R/(xyz)) = \eh(I R/(x)) + \eh(I R/(y)) + \eh(I R/(z)) = bc + ac + ab.
\]

For computing $\eh_1  (\m \mid I)=\eh(I_1)$, we take a general element of the form $\alpha=\lambda_1 x + \lambda_2y-z \in \m$. 
After a possible change of variables, we may write $I_1 = (x^a, y^b, (x+y)^c, xy(x + y))R_1$, where $R_1 = \textsf{k}[[x,y]]$. 
Since $\alpha \in \m$ is general, the element $xy(x + y) \in I_1$ is still superficial, so by the additivity property
\[
\eh(I_1)=\eh \left( (x^a, y^b, (x+y)^c) \left( \textsf{k}[x,y]/(xy(x+y)) \right) \right)
= \eh ((y^b)\textsf{k}[y]) + 2\eh((x^a)\textsf{k}[x]) = 2a + b.
\]
Finally, note that $\eh_2(\m \mid I) =  \ord(I)=3$.
\end{proof}

\begin{ex}
Let $I$ be the ideal described in Proposition \ref{EX}.
Then Conjecture~\ref{mu conj} asserts that 
$$2 \mu(I) \geq \eh_1 (\m \mid I) + 3 \eh_2 (\m \mid I) + 2 \eh_3 (\m \mid I),$$ 
which becomes
\[
4a + 2b +2 \geq 2a+ b + 9 + 2 = 2a + b + 11.
\]
Thus we have equality if and only if $a = b = 3$. In this case $I_1 = (x, y)^3$.\\

For Conjecture~\ref{length conj} we need that 
\[
6 \len (R/I) \geq \eh(I) + 3 \eh(I_1) + 2 \eh(I_2) = ab + bc + ac + 6a +3b + 6.
\]
However, 
\begin{align*}
6 \len (R/I) &\geq
6 \left  (
\left \lceil \frac{ab + b}{2}\right\rceil + 
\left \lceil \frac{ac + c}{2}\right\rceil + \left \lceil \frac{bc + c}{2} \right\rceil - a - b - c + 1
\right )\\
&\geq 3 (ab + bc + ac)  - 3a - 3b - 3c + 6,
\end{align*}
so it remains to show that
\[
2ab + 2bc + 2ac \geq 9a + 6b + 3c.
\]
This can be seen by setting $b = a + x$ and $c = a + y$. Then on the left side we have
\[
ab + bc + ac  = 3a^2 + 2ax + 2ay + xy,
\]
and the inequality becomes
\[
6a^2 + 4ax + 4ay + 2xy \geq 18 a + 6x + 3y.
\]
This inequality is always true since $a \geq 3$ by  assumption, so $6a^2 \geq 18a$, $4ax \geq 6x$, and $4ay \geq 3y$.
We also note that equality holds if and only if $a = b = c = 3$, that is $I = \m^3$.
\end{ex}

\bigskip


\end{document}